\documentclass[11pt,reqno]{amsart}
\usepackage[T1]{fontenc}
\usepackage[utf8]{inputenc}
\usepackage[a4paper,text={160mm,240mm},centering]{geometry}
\usepackage{amsmath,amssymb,amsthm,mathrsfs}
\usepackage{xcolor}
\usepackage{microtype}
\usepackage[hidelinks]{hyperref}

\newif\ifshowrevision
\showrevisionfalse
\newcommand{\rev}[1]{{\ifshowrevision\color{red}\fi #1}}
\newenvironment{revision}{\par\begingroup\ifshowrevision\color{red}\fi}{\par\endgroup}
\definecolor{reviewgreen}{RGB}{0,112,48}
\newcommand{\greenrev}[1]{{\ifshowrevision\color{reviewgreen}\fi #1}}
\newenvironment{greenrevision}{\par\begingroup\ifshowrevision\color{reviewgreen}\fi}{\par\endgroup}
\newtheorem{theorem}{Theorem}[section]
\newtheorem{lemma}[theorem]{Lemma}
\newtheorem{prop}[theorem]{Proposition}
\newtheorem{corollary}[theorem]{Corollary}
\newtheorem{conj}[theorem]{Conjecture}
\theoremstyle{definition}
\newtheorem{definition}[theorem]{Definition}
\newtheorem{remark}[theorem]{Remark}
\newtheorem{example}[theorem]{Example}
\newtheorem*{ack}{Acknowledgements}

\DeclareMathOperator{\Ann}{Ann}
\DeclareMathOperator{\tr}{tr}
\DeclareMathOperator{\gr}{gr}
\DeclareMathOperator{\ord}{ord}
\DeclareMathOperator{\Der}{Der}
\DeclareMathOperator{\Soc}{Soc}

\NewMarkClass{paper-section}
\makeatletter
\let\paper@tocwrite\@tocwrite
\renewcommand{\@tocwrite}[2]{%
  \paper@tocwrite{#1}{#2}%
  \def\paper@section{section}\def\paper@kind{#1}%
  \ifx\paper@kind\paper@section
    \ifx\@secnumber\@empty
      \InsertMark{paper-section}{#2}%
    \else
      \InsertMark{paper-section}{\thesection.\enspace#2}%
    \fi
  \fi
}

\makeatother

\allowdisplaybreaks
\title[Euler conductors]{The Hassanzadeh--Nasrollah Nejad--Simis Conjecture on Euler Conductors}

\author{Yizhi Zhang}
\address{Zhili College, Tsinghua University, Beijing, 100084, P. R. China}
\email{yizhi-zh24@mails.tsinghua.edu.cn}

\author{Huaiqing Zuo}
\address{Department of Mathematical Sciences, Tsinghua University, Beijing, 100084, P. R. China}
\email{hqzuo@mail.tsinghua.edu.cn}

\begin{document}
\begin{abstract}
{\color{black} We prove the Hassanzadeh--Nasrollah Nejad--Simis conjecture:
if {\thickmuskip=3mu $f\in (x_1,\ldots,x_n)^2\mathrel{\textcolor{black}{\subsetneq}} k[[x_1,\ldots,x_n]]$} has an
isolated critical point and $\operatorname{char}k=0$, then its Euler
conductor $J_f:f$ is not contained in the Tjurina ideal $(J_f,f)$.
More generally, if $A=R/I$ is a nonzero Noetherian $k$-algebra and
$a=[f]\in A$ is nilpotent with $d_{A/k}a=0$, then
$I:f\nsubseteq(I,f)$. For a local ring $(R,\mathfrak n)$ we obtain
the stronger noncontainment
$I:f\nsubseteq(I,f)+\mathfrak n(I:f)$.
The proof reduces a hypothetical containment to a self-exact
square-zero element and detects its differential by the trace of a
regular representation over the dual numbers. 
Finally, an explicit five-variable isolated singularity satisfies
$J_f:f\subseteq\overline{J_f}$, disproving the integral-closure
strengthening proposed by Ma and Zuo.}

Keywords.  Euler conductor, K\"ahler differentials, 
isolated hypersurface singularity, integral closure.
		
		MSC(2020).  Primary 14B05; Secondary 13N05, 32S05.
\end{abstract}
\maketitle

\begin{revision}
\begingroup
\showrevisionfalse
\color{black}
\section{Introduction}\label{sec:introduction}

Let $k$ be a field of characteristic zero, put
$R=k[[x_1,\ldots,x_n]]$ and $\mathfrak m=(x_1,\ldots,x_n)$, and
let $f\in\mathfrak m^2$ have an isolated critical point, meaning that
its Jacobian ideal
\[
J_f=(\partial_1f,\ldots,\partial_nf)
\]
is $\mathfrak m$-primary. The Tjurina ideal is $I_f=(J_f,f)$.
The Milnor and Tjurina numbers are
$\mu_f=\dim_k(R/J_f)$ and $\tau_f=\dim_k(R/I_f)$.
For background on these invariants and their deformation-theoretic
meaning, see \cite{GLS}; the topological interpretation of $\mu_f$
for complex analytic germs is due to Milnor \cite{Mil}.
The exact sequence
\[
0\longrightarrow [f](R/J_f)\longrightarrow R/J_f
\longrightarrow R/I_f\longrightarrow0
\]
shows that $\mu_f-\tau_f=\dim_k[f](R/J_f)$. In particular,
$\mu_f=\tau_f$ if and only if $f\in J_f$. For an isolated complex
analytic hypersurface germ, Saito's criterion identifies this last
condition with weighted homogeneity after an analytic change of
coordinates \cite{Sai}.

The numerical comparison of $\mu_f$ and $\tau_f$ has a substantial
literature; see, for example, \cite{Alm,BH,DG,Liu,Wat} and the
recent work \cite{MZ}. A complementary question concerns the position
of the \emph{Euler conductor}
\[
J_f:f=\{g\in R:gf\in J_f\}
\]
relative to $I_f$. Hassanzadeh, Nasrollah Nejad, and Simis proposed
the following statement and proved it for plane curves
\cite[Conjecture 3.6 and Proposition 3.7]{HNS}.

\begin{conj}[Hassanzadeh--Nasrollah Nejad--Simis]\label{conj1}
Under the preceding assumptions,
\begin{equation}\label{eq:intro-original}
J_f:f\nsubseteq(J_f,f).
\end{equation}
\end{conj}

We prove this conjecture in every dimension. The key is to pass from
ideals to a nilpotent element of a quotient algebra. For $A=R/J_f$
and $a=[f]$, one has
\[
(J_f:f)/J_f=\Ann_A(a),\qquad (J_f,f)/J_f=aA.
\]
Thus \eqref{eq:intro-original} asks whether
$\Ann_A(a)\nsubseteq aA$. In the finite-dimensional setting this
means that multiplication by $a$ has at least one Jordan block of
size one. Nilpotence alone does not imply this: in
$k[t]/(t^2)$, the annihilator of $t$ equals $tA$.
What distinguishes a Jacobian quotient is the identity
$d_{A/k}a=0$ in its module of ordinary K\"ahler differentials.

\begin{theorem}\label{thm:intro-main}
Let $k$ be a field of characteristic zero, let $R$ be a commutative
$k$-algebra, and let $I\subsetneq R$ be an ideal such that $A=R/I$
is Noetherian. If $a=[f]\in A$ is nilpotent and $d_{A/k}a=0$, then
\[
I:f\nsubseteq(I,f).
\]
\end{theorem}

\begin{theorem}\label{thm:intro-local}
Under the assumptions of Theorem~\ref{thm:intro-main}, if
$(R,\mathfrak n)$ is local, then
\[
I:f\nsubseteq(I,f)+\mathfrak n(I:f).
\]
\end{theorem}

Theorems~\ref{thm:intro-main} and \ref{thm:intro-local} are proved
as Theorems~\ref{thm:noetherian-colon} and
\ref{thm:local-strengthening}. Only the quotient $R/I$, not $R$
itself, is required to be Noetherian. If $R$ is Noetherian local,
the second theorem gives an element of $I:f$ that can be included
in a minimal generating set and whose class remains nonzero modulo
$(I,f)+\mathfrak n(I:f)$; see Remark~\ref{rem:minimal-generator}.

The argument has two main steps. First, a hypothetical containment
$\Ann_A(a)\subseteq aA$ produces a nonzero element $c$ in
$C=A/(a^2)$ with $c^2=0$ and $\Ann_C(c)=cC$. Second, we prove the
conormal obstruction
\[
d_{C/k}c\notin c\Omega_{C/k}.
\]
For a finite-dimensional algebra this follows from a differential
trace over the dual numbers. Localization at a minimal prime and
a coefficient-field lifting argument reduce the Noetherian case
to that computation. For the terminology of exact zero divisors, see \cite{HenSeg}.
No Gorenstein or complete-intersection hypothesis enters this obstruction.

For formal power series, verifying $d_{A/k}[f]=0$ requires care:
ordinary K\"ahler differentials cannot simply be identified with
continuous differentials. Lemma~\ref{lem:finite-jet} supplies a
polynomial-truncation argument when the critical quotient has
finite length. This yields the strengthened Jacobian statement
\[
J_f:f\nsubseteq(J_f,f)+\mathfrak m(J_f:f).
\]
For smooth algebras of finite type, differential vanishing is
immediate from the cotangent module. At a rational critical point,
subtracting the critical value also makes the resulting class
nilpotent, even when the critical locus is positive-dimensional.

There is an important limitation. Ma and Zuo proposed the stronger
statement below \cite[Conjecture 3.10]{MZ}.
\begin{conj}[Integral-closure strengthening]\label{conj2}
\greenrev{Let $K$ be a field, with no restriction on its characteristic,
and put $S=K[[x_1,\ldots,x_n]]$.
Suppose $f\in(x_1,\ldots,x_n)^2$ and
$\dim_K S/(J_f,f)<\infty$.
If $f\in\overline{J_f}$, then $J_f:f\nsubseteq\overline{J_f}$.}
\end{conj}
\begin{greenrevision}
Here the finite-dimensional Tjurina quotient expresses the
isolated-hypersurface hypothesis used in \cite[Conjecture 3.10]{MZ}.
Under the additional assumption $f\in\overline{J_f}$, it is
equivalent to the isolated-critical-point condition:
an integral equation for $f$ gives $f^s\in J_f$ for some $s$,
so $f\in\sqrt{J_f}$ and
\[
\sqrt{(J_f,f)}=\sqrt{J_f}.
\]
Thus $(J_f,f)$ is primary to the maximal ideal if and only if
$J_f$ is. This observation is valid in arbitrary characteristic.
\end{greenrevision}
The hypothesis $f\in\overline{J_f}$ implies
$(J_f,f)\subseteq\overline{J_f}$; in characteristic zero this
integrality hypothesis is automatic
\greenrev{\cite[Theorem 7.1.5 and Corollary 7.2.6]{HS}}.
Nevertheless, Conjecture~\ref{conj2} is false even in characteristic
zero. We prove that
\[
f=\frac14\sum_{i=1}^5x_i^4+(x_1+\cdots+x_5)^5
\quad\text{in }\mathbb Q[[x_1,\ldots,x_5]]
\]
has an isolated critical point and satisfies
$J_f:f\subseteq\overline{J_f}=\mathfrak m^3$.
The needed injectivity in a monomial complete intersection is a
special case of the strong Lefschetz theorem
\cite{Stanley,Watanabe}; we give its representation-theoretic
verification explicitly.

Section~\ref{sec:background} fixes the conventions and recalls the
algebraic background. Section~\ref{sec:finite-trace} gives the
Jordan reformulation and the finite-dimensional trace lemma.
Sections~\ref{sec:noether-rigidity} and \ref{sec:colon-theorems}
prove the conormal obstruction and the two ideal quotient
theorems. Section~\ref{sec:applications} treats Jacobian ideals and
critical loci, and Section~\ref{sec:integral-counterexample}
constructs the integral-closure counterexample.

\begin{ack}
The authors are grateful to AI models, particularly ChatGPT 6 Astra, for valuable assistance in exploring some of the proof strategies and calculations. The authors take full responsibility for the human verification and exposition of the results.    H. Zuo was supported by NSFC Grant 12271280 and BJNSF Grant 1252009.

\end{ack}

\endgroup 
\begingroup
\showrevisionfalse
\color{black}
\section{Preliminaries and conventions}\label{sec:background}

\textcolor{black}{We introduce, in turn, the basic isolated-singularity setting, the numerical and integral-closure background, and the differential tools needed for the main proof. The complete-intersection and Gorenstein properties are recalled only for context; the main proof does not rely on them.}

All rings are commutative with identity, all algebra maps preserve
the identity, and all fields have characteristic zero unless
explicitly stated otherwise. No finite-type hypothesis is implicit
in the term $k$-algebra. For an ideal $I\mathrel{\textcolor{black}{\subseteq}} R$, write
$(I,f)=I+fR$ and $I:f^0=I$. An ideal is \emph{$\mathfrak m$-primary}
if its radical is the maximal ideal $\mathfrak m$; in the
Noetherian local setting its quotient then has finite length.
For a finite-length $R$-module $M$, $\lambda_R(M)$ denotes its
 length. If $R$ has residue field $k$ and contains that
field, this is also $\dim_kM$.

\subsection{Singularities and weighted homogeneity}
In the formal hypersurface setting we always impose
$f\in\mathfrak m^2$ and $J_f$ $\mathfrak m$-primary. In particular,
the critical value is zero and the critical quotient is nonzero.
These assumptions must not be replaced just by the condition that
$R/J_f$ is Artinian: adding a nonzero constant to $f$ does not
change $J_f$, but does change the class of $f$ in that quotient.

\begin{definition}\label{quasi}
A nonzero polynomial $F\in k[x_1,\ldots,x_n]$ is
\emph{weighted homogeneous} of weighted degree $d$ for positive
rational weights $w_1,\ldots,w_n$ if each monomial $x^\alpha$
with nonzero coefficient in $F$ satisfies
$\sum_iw_i\alpha_i=d$. A formal power series is
\emph{formally right equivalent} to another if a continuous
$k$-algebra automorphism of $k[[x_1,\ldots,x_n]]$ transforms one
into the other. A germ is \emph{quasihomogeneous} if it is right
equivalent to a weighted homogeneous polynomial, using analytic
coordinate changes in the analytic category.
\end{definition}

\greenrev{After multiplying all weights and $d$ by a common denominator,
the weights may be taken to be positive integers and $d$ to be
a nonnegative integer. The degree is positive when $F$ is
nonconstant; a nonzero constant has degree zero.} The monomial definition is
equivalent to
\[
F(t^{w_1}x_1,\ldots,t^{w_n}x_n)=t^dF(x_1,\ldots,x_n)
\quad\text{and}\quad
dF=\sum_iw_ix_i\partial_iF.
\]

Indeed, the first identity is checked coefficientwise, and the
coefficient of $x^\alpha$ in the second is multiplied by
$\sum_iw_i\alpha_i$; in characteristic zero it equals $d$ precisely
for the prescribed monomials. We do not impose an unnecessary
upper bound on individual weights. Saito's theorem \cite{Sai}
states that, for an isolated germ in $\mathbb C\{x_1,\ldots,x_n\}$,
quasihomogeneity is equivalent to $f\in J_f$. Our main proof uses
the latter algebraic condition when appropriate and does not
require a descent version of Saito's theorem over arbitrary fields.

\textcolor{black}{The preceding discussion concerns the special case of quasihomogeneity. We now recall the structure of the quotient algebra that holds for every isolated critical point.}

\begin{greenrevision}
Here a sequence $g_1,\ldots,g_r$ is \emph{regular} if
$(g_1,\ldots,g_r)\ne R$ and each $g_i$ is a nonzerodivisor
in the quotient ring $R/(g_1,\ldots,g_{i-1})$.

In the preceding isolated-critical-point setting, $J_f$ is
$\mathfrak m$-primary. Hence its $n$ partial-derivative generators
$\partial_1f,\ldots,\partial_nf$ form a system of parameters
in the $n$-dimensional regular local ring $R$.
Since a regular local ring is Cohen--Macaulay, this system of
parameters is a regular sequence. This conclusion does not
require $f$ to be quasihomogeneous.

Consequently, putting $A=R/J_f$, we obtain an Artinian local
complete intersection, which is in particular Gorenstein.
Write $\mathfrak m_A=\mathfrak m/J_f$ for its maximal ideal.
Since its residue field is $k$, the Gorenstein property is
equivalent to
\[
\dim_k\Soc(A)=1,\qquad
\Soc(A)=\Ann_A(\mathfrak m_A).
\]
These standard facts are recalled only for context; the proofs
below do not rely on these complete-intersection or Gorenstein
properties; see \cite{Eis,GLS}.
\end{greenrevision}

\subsection{Integral closure and nilpotence exponents}
\begin{definition}\label{def:integral-closure}
An element $u\in R$ is \emph{integral over} an ideal $Q$ if
$u^s+b_1u^{s-1}+\cdots+b_s=0$ for some $s\geq1$ and
$b_i\in Q^i$. The set of such elements is the integral closure
$\overline Q$. The ideal $Q$ is integrally closed if
$Q=\overline Q$. An ideal $J\subseteq Q$ is a \emph{reduction}
of $Q$ if $Q^{r+1}=JQ^r$ for some $r\geq0$.
In a Noetherian ring a reduction has the same integral closure as
the ideal it reduces; see \cite{HS}.
\end{definition}

\begin{definition}\label{BSe}
For $f\in\mathfrak m^2$ with $J_f$ $\mathfrak m$-primary, set
\[
e^{\mathrm{BS}}(f)=\min\{d\geq1:f^d\in J_f\}.
\]
This is the nilpotence index of $[f]$ in $R/J_f$, with value one
when $f\in J_f$.
\end{definition}
\begin{greenrevision}
Existence follows from the nilpotence of the maximal ideal of
the Artinian ring $R/J_f$. To obtain a dimension bound, we use
the following form of the Brian\c{c}on--Skoda theorem:
if $Q$ is generated by $r\geq1$ elements in a regular local ring
containing a field of characteristic zero, then
\[
\overline{Q^{r+q-1}}\subseteq Q^q\qquad(q\geq1);
\]
see \cite{Sko,LS} and \cite[Theorem 13.3.3]{HS}.
For a characteristic-zero formal power series with zero constant
term, one has $f\in\overline{J_f}$ by
\cite[Corollary 7.2.6]{HS}.
The valuative criterion for integral closure then gives
\[
v(f^n)=n\,v(f)\geq n\,v(J_f)=v(J_f^n)
\]
for every valuation in that criterion, where $v(Q)$ is the
minimum of the values of a finite set of generators of $Q$.
Consequently $f^n\in\overline{J_f^n}$.
Applying the displayed theorem in $R=k[[x_1,\ldots,x_n]]$
with $Q=J_f$, $r=n$, and $q=1$ yields
\[
f^n\in\overline{J_f^n}\subseteq J_f,
\qquad e^{\mathrm{BS}}(f)\leq n.
\]
This bound is not needed to prove the conductor noncontainment.
\end{greenrevision}

\begin{definition}\label{def:beta}
Following \cite{HNS}, define
\[
\beta_f=\max\{i\geq0:(J_f:f^i)+(f)=I_f\}.
\]
The maximum exists: $i=0$ satisfies the equality, whereas
$J_f:f^{e^{\mathrm{BS}}(f)}=R$ and $I_f\ne R$.
The ascending chain  $\cdots \mathrel{\textcolor{black}{\subseteq}} (J_f:f^i)\mathrel{\textcolor{black}{\subseteq}}  (J_f:f^{i+1})\mathrel{\textcolor{black}{\subseteq}} \cdots $ implies that Conjecture \ref{conj1} is equivalent to $\beta_f=0$.
\end{definition}

\subsection{Derivations and conormal modules}
For a $k$-algebra $B$ and a $B$-module $M$, a \emph{$k$-derivation}
is an additive map $\delta:B\to M$ vanishing on $k$ and satisfying
$\delta(xy)=x\delta(y)+y\delta(x)$.
The module of K\"ahler differentials $\Omega_{B/k}$ is generated
over $B$ by symbols $dx$, $x\in B$, with exactly these additive,
$k$-constant, and Leibniz relations. Its universal derivation
$d_{B/k}:B\to\Omega_{B/k}$ identifies
\[
\operatorname{Hom}_B(\Omega_{B/k},M)
\simeq\Der_k(B,M).
\]
Write $\Der_k(B)=\Der_k(B,B)$. 

For an ideal $H\subseteq B$, the \emph{conormal module} is
$H/H^2$, and the conormal sequence is right exact:
\[
H/H^2\longrightarrow\Omega_{B/k}\otimes_B B/H
\longrightarrow\Omega_{(B/H)/k}\longrightarrow0,
\qquad h\bmod H^2\longmapsto dh\otimes1.
\]
The first map need not be injective. Theorems below detect a
particular nonzero conormal image. Differentials commute with
localization, and for $k\subseteq L\subseteq E$ the transitivity
sequence gives a surjection
$\Omega_{E/k}\twoheadrightarrow\Omega_{E/L}$.
These facts and the defining universal property can be found in
\cite[Tag 00RM]{Stacks}.
\endgroup 
\end{revision}

\section{Exact zero divisors and the finite-dimensional trace}
\label{sec:finite-trace}

\begin{greenrevision}
Return to the isolated-singularity setting of Section~\ref{sec:background}.
Let $k$ be a field of characteristic zero. Set
$R=k[[x_1,\ldots,x_n]]$ and $\mathfrak m=(x_1,\ldots,x_n)$. Let $f\in\mathfrak m^2$
have an $\mathfrak m$-primary Jacobian ideal $J_f$. Put
\[
A=R/J_f,\qquad a=[f]\in A.
\]
Here $A$ is a finite-dimensional Artinian local $k$-algebra, and $[g]$
denotes the class of $g$ modulo $J_f$.
Write $\pi:R\to A$, $g\mapsto[g]$, for the natural quotient map.

\end{greenrevision}
\rev{For an element $b$ of a ring $B$, its annihilator is the ideal
$\Ann_B(b)=\{u\in B:ub=0\}$.}

\begin{greenrevision}
\begin{prop}\label{prop:quotient-equivalence}
In the preceding setting,
\begin{equation}\label{eq:colon-dictionary-background}
(J_f:f)/J_f=\Ann_A(a),\qquad (J_f,f)/J_f=aA.
\end{equation}
Consequently,
\[
J_f:f\nsubseteq(J_f,f)
\quad\Longleftrightarrow\quad
\Ann_A(a)\nsubseteq aA.
\]
\end{prop}
\end{greenrevision}
\begin{proof}
\begin{greenrevision}
Write $K=J_f:f$ and $H=(J_f,f)=J_f+fR$. Since $J_f$ is an ideal,
every $g\in J_f$ satisfies $gf\in J_f$, so $J_f\subseteq K$;
also, $J_f\subseteq H$. The natural quotient map $\pi:R\to A$
has kernel $J_f$, and hence its restrictions to $K$ and $H$
both have kernel $J_f$. Thus, via $g+J_f\mapsto[g]$,
we identify $K/J_f$ and $H/J_f$ with the ideals $\pi(K)$ and $\pi(H)$ of $A$, respectively.

First, for every $g\in R$,
\[
[g]\in\Ann_A(a)
\quad\Longleftrightarrow\quad [g][f]=0
\quad\Longleftrightarrow\quad gf\in J_f
\quad\Longleftrightarrow\quad g\in K.
\]
The middle equivalence follows from $[g][f]=[gf]$ and $\ker\pi=J_f$.
Since $\pi$ is surjective, every element of $A$ has the form $[g]$.
The displayed equivalences therefore give $\pi(K)=\Ann_A(a)$,
proving the first equality.

Second, every element of $H$ has the form $q+fh$, with $q\in J_f$ and $h\in R$,
and $\pi(q+fh)=a[h]\in aA$. Conversely, for any $u\in A$,
surjectivity of $\pi$ gives $h\in R$ with $u=[h]$, and then
$au=\pi(fh)\in\pi(H)$. Hence $\pi(H)=aA$, proving the second equality.
The same argument also gives
\[
[g]\in aA
\quad\Longleftrightarrow\quad
g-fh\in J_f\text{ for some }h\in R
\quad\Longleftrightarrow\quad g\in H.
\]

Finally, we verify the equivalence of the two containments explicitly.
Suppose $K\subseteq H$. For any $u\in\Ann_A(a)$, choose $g\in K$
with $u=[g]$. Since $g\in H$, we have $u\in\pi(H)=aA$,
so $\Ann_A(a)\subseteq aA$. Conversely, suppose $\Ann_A(a)\subseteq aA$.
For every $g\in K$, we have $[g]\in\Ann_A(a)\subseteq aA$.
The preceding equivalences give $g\in H$, and thus $K\subseteq H$.
Consequently,
\postdisplaypenalty=10000
\[
J_f:f\subseteq(J_f,f)
\quad\Longleftrightarrow\quad
\Ann_A(a)\subseteq aA.
\]
\nopagebreak[4]
Negating both sides gives the equivalence of noncontainments in the proposition. \qedhere
\end{greenrevision}
\end{proof}

\begin{greenrevision}
We first give the Jordan reformulation and block-count formula,
then reduce a hypothetical containment to a self-exact square-zero element,
and finally establish the differential obstruction in the finite-dimensional case.

\end{greenrevision}
\begin{greenrevision}
\begin{prop}\label{prop:jordan-reformulation}
In the preceding setting, $a$ is nilpotent. Let $m_a:A\to A$ denote
the $k$-linear map $x\mapsto ax$, and let $b_1$ be the number of its
Jordan blocks of size one. Then
\[
\Ann_A(a)\not\subseteq aA
\quad\Longleftrightarrow\quad
m_a\text{ has a Jordan block of size one}.
\]
\begin{equation}\label{eq:jordan-count}
b_1=\dim_k\frac{\Ann_A(a)}{\Ann_A(a)\cap aA}
=2\dim_k(A/aA)-\dim_k(A/a^2A).
\end{equation}
\end{prop}
\end{greenrevision}
\begin{proof}
\begin{greenrevision}
We first prove the Jordan reformulation. Since $m_a(x)=ax$,
we have $\ker m_a=\Ann_A(a)$ and $\operatorname{im}m_a=aA$.
Thus we must determine whether $\ker m_a$ is contained in $\operatorname{im}m_a$.

Let $V$ be the invariant subspace corresponding to a Jordan block of size $s$ for $m_a$.
It has a basis $v_1,\ldots,v_s$ satisfying
\[
m_a(v_1)=0,\qquad m_a(v_i)=v_{i-1}\quad(2\le i\le s).
\]
For $x=\sum_{i=1}^s\alpha_i v_i\in V$, we have
$m_a(x)=\sum_{i=2}^s\alpha_i v_{i-1}$.
Linear independence of the basis vectors shows that $m_a(x)=0$
if and only if $\alpha_2=\cdots=\alpha_s=0$, so $\ker(m_a|_V)=kv_1$.

If $s=1$, the restriction of $m_a$ to $V$ is the zero map.
Hence $\ker(m_a|_V)=V\ne0$, whereas $m_a(V)=0$;
in this case the kernel is not contained in the image.
If $s\ge2$, then $m_a(V)=\operatorname{span}_k\{v_1,\ldots,v_{s-1}\}$.
In particular, $v_1=m_a(v_2)$ lies in the image, and therefore
$\ker(m_a|_V)=kv_1\subseteq m_a(V)$.

Since $A$ is the direct sum of these invariant subspaces,
the kernel and image of $m_a$ are the direct sums of the kernels and images on the individual blocks.
Thus $\ker m_a\subseteq\operatorname{im}m_a$ holds if and only if
the kernel on every block is contained in its image, or equivalently,
every block has size at least two. Negating this equivalence gives
$\Ann_A(a)\nsubseteq aA$ if and only if there is a Jordan block of size one.
\end{greenrevision}
\begin{greenrevision}
To verify \eqref{eq:jordan-count}, retain the preceding Jordan decomposition
\[
A=\bigoplus_{\nu=1}^{t}V_\nu,
\]
where $s_\nu$ is the size of $V_\nu$. This is a decomposition of $k$-vector spaces;
the $V_\nu$ need not be ideals of $A$. On a block $V$ of size $s$, the same basis gives
\[
m_a^2(V)=\operatorname{span}_k\{v_1,\ldots,v_{s-2}\},
\]
with this span understood to be zero when its upper index is less than one.
Thus $V/m_a(V)$ has basis $[v_s]$ and dimension one;
if $s\ge2$, then $V/m_a^2(V)$ has basis $[v_{s-1}],[v_s]$ and dimension two,
whereas its dimension is one when $s=1$. Hence
\[
\dim_k\bigl(V/m_a(V)\bigr)=1,\qquad
\dim_k\bigl(V/m_a^2(V)\bigr)=\min\{s,2\}.
\]
The preceding kernel and image calculation also shows that
$\ker(m_a|_V)/(\ker(m_a|_V)\cap m_a(V))$ has dimension one when $s=1$
and is zero when $s\ge2$.

Since $m_a$ preserves every block,
\[
\begin{aligned}
aA&=m_a(A)=\bigoplus_\nu m_a(V_\nu),\\
a^2A&=m_a^2(A)=\bigoplus_\nu m_a^2(V_\nu),\\
\Ann_A(a)&=\ker m_a=\bigoplus_\nu\ker(m_a|_{V_\nu}).
\end{aligned}
\]
Moreover, the direct-sum decomposition gives
\[
\ker m_a\cap m_a(A)
=\bigoplus_\nu\bigl(\ker(m_a|_{V_\nu})\cap m_a(V_\nu)\bigr).
\]
Consequently, all three quotients in \eqref{eq:jordan-count} decompose blockwise, and
\[
\begin{aligned}
\dim_k\frac{\Ann_A(a)}{\Ann_A(a)\cap aA}&=b_1,\\
\dim_k(A/aA)&=t,\\
\dim_k(A/a^2A)&=\sum_{\nu=1}^{t}\min\{s_\nu,2\}.
\end{aligned}
\]
Finally, each block of size one contributes one to $2\dim_k(A/aA)-\dim_k(A/a^2A)$,
and each block of size at least two contributes zero. Therefore
\[
2\dim_k(A/aA)-\dim_k(A/a^2A)
=\sum_{\nu=1}^{t}\bigl(2-\min\{s_\nu,2\}\bigr)=b_1,
\]
which proves \eqref{eq:jordan-count}. \qedhere
\end{greenrevision}
\end{proof}

\begin{greenrevision}
The preceding proposition gives the linear-algebraic form of the original problem.
Nilpotence alone does not exclude the possibility that all blocks have size at least two;
the additional constraint is the differential identity $d_{A/k}a=0$.
In the present Jacobian quotient setting, this identity is proved by the finite polynomial
truncation argument in Lemma~\ref{lem:finite-jet}.
We first establish the differential obstruction in the finite-dimensional case,
and subsequently extend it to arbitrary Noetherian quotient rings.
\end{greenrevision}

\rev{For a commutative $L$-algebra $B$,}
write
\[
d_{B/L}:B\longrightarrow\Omega_{B/L}
\]
for the universal K\"ahler derivation. When the algebra and base field
are determined by context, we abbreviate \(d_{B/L}(x)\) to \(dx\). Thus,
in Lemma~\ref{lem:finite-trace}, \(dc\) means \(d_{E/L}(c)\), whereas
\(d\varepsilon\) means \(d_{D/L}(\varepsilon)\); in
Theorem~\ref{thm:conormal-rigidity}, \(dc\) means \(d_{C/k}(c)\). For a
matrix \(X=(x_{ij})\) with entries in \(D\), set
\(dX=(d_{D/L}x_{ij})\), taking the differential entrywise.

\begin{revision}
\begin{definition}\label{def:exact-zero-divisor}
A pair of nonzero nonunits $(u,v)$ in a ring $B$ is an
\emph{exact pair of zero divisors} if
$\Ann_B(u)=vB$ and $\Ann_B(v)=uB$; see \cite[Section 1]{HenSeg}.
We call $c$ \emph{self-exact square-zero} if
$c\ne0$, $c^2=0$, and $\Ann_B(c)=cB$.
Equivalently, the complex with every map multiplication by $c$,
\[
\cdots\xrightarrow{c}B\xrightarrow{c}B\xrightarrow{c}B
\xrightarrow{c}\cdots ,
\]
is exact. For $H=cB$, the conormal module is $H/H^2=cB$, and
$B/cB\to cB$, $\bar x\mapsto cx$, is an isomorphism.
\end{definition}
\end{revision}

\begin{lemma}\label{lem:colon-exact}
Let \(A\) be a commutative ring and let \(0\ne a\in A\) be nilpotent.
If \(\Ann_A(a)\subseteq aA\), then
\[
\Ann_A(a)=a\Ann_A(a^2).
\]
Consequently, if \(C=A/(a^2)\) and \(c=a\bmod(a^2)\), then
\begin{equation}\label{eq:self-exact}
c\ne0,\qquad c^2=0,\qquad\Ann_C(c)=cC.
\end{equation}
\end{lemma}
\begin{proof}
If \(x\in\Ann_A(a)\), the assumption gives \(x=ay\), and then
\(a^2y=0\). Hence
\(\Ann_A(a)\subseteq a\Ann_A(a^2)\). Conversely, if
\(x=ay\in a\Ann_A(a^2)\), then \(a^2y=0\), so
\(ax=a^2y=0\). Thus \(x\in\Ann_A(a)\), proving the reverse inclusion.

If \(c=0\), then
\(a=a^2u\), so \(a(1-au)=0\). Since \(au\) is nilpotent, \(1-au\) is a
unit, contradicting \(a\ne0\). If \(c\bar x=0\), lift \(\bar x\) to
\(x\in A\). There is a \(y\) with \(ax=a^2y\), and hence
\[
x-ay\in\Ann_A(a)=a\Ann_A(a^2).
\]
Reduction modulo \(a^2A\) yields \(\bar x\in cC\). The reverse
inclusion can be checked directly: if \(\bar x=c\bar y\), then
\(c\bar x=c^2\bar y=0\), so \(\bar x\in\Ann_C(c)\). \qedhere
\end{proof}

\begin{greenrevision}
Lemma~\ref{lem:colon-exact} reduces the hypothetical containment to the existence
of a self-exact square-zero element. We now study the differential of such an element
in a finite-dimensional algebra. Sending $\varepsilon$ in the dual-number algebra to $c$
makes the finite-dimensional algebra containing $c$ a module over the dual numbers.
The next lemma shows that this module is free, so multiplication can be represented
by matrices, preparing the construction of the trace map below.

The following results apply to any finite-dimensional commutative algebra $E$
over a field $L$ of characteristic zero. In the contradiction argument for the preceding
Jacobian quotient, one can take $L=k$ and $E=C=A/(a^2)$.

\begin{lemma}\label{lem:dual-freeness}
Let $L$ be a field of characteristic zero, let $E$ be a finite-dimensional
commutative $L$-algebra, and suppose that $0\ne c\in E$ satisfies
$c^2=0$ and $\Ann_E(c)=cE$. Put $D=L[\varepsilon]/(\varepsilon^2)$
and make $E$ a $D$-algebra via $\varepsilon\mapsto c$.
Then $D\to E$ is injective, and $E$ is a free $D$-module of rank
$r=\dim_L(cE)>0$.
\end{lemma}
\end{greenrevision}
\begin{proof}
\begin{greenrevision}
Put $D=L[\varepsilon]/(\varepsilon^2)$ and let $D$ act on $E$
by sending $\varepsilon$ to $c$.
The induced map $D\to E$ is injective. Indeed, a relation
$\alpha+\beta c=0$ with $\alpha,\beta\in L$ gives
$\alpha=\beta=0$: if $\beta\ne0$, then $c$ would be a scalar,
whereas the only nilpotent scalar is zero.

Let $m_c$ denote multiplication by $c$. The assumptions give
\[
\ker(m_c)=cE=\operatorname{im}(m_c).
\]
Set $r=\dim_L(cE)>0$ and choose elements $u_1,\ldots,u_r$
such that $cu_1,\ldots,cu_r$ form an $L$-basis of $cE$.
For any $x\in E$, write $cx=\sum_i\alpha_i cu_i$.
Then
\[
x-\sum_i\alpha_i u_i\in\ker(m_c)=cE,
\]
so there exist $\beta_i\in L$ with
\[
x=\sum_i(\alpha_i+\beta_i c)u_i.
\]
Thus the $u_i$ generate $E$ as a $D$-module. If the right-hand
side is zero, multiplication by $c$ yields
$\sum_i\alpha_i cu_i=0$, hence $\alpha_i=0$ for every $i$.
The remaining relation $\sum_i\beta_i cu_i=0$ then forces
every $\beta_i=0$. Therefore the $u_i$ form a $D$-basis and
$E\simeq D^r$. \qedhere
\end{greenrevision}
\end{proof}

\begin{greenrevision}
The following lemma provides the trace map
and its properties needed in the main proof.

\begin{lemma}\label{lem:differential-trace-map}
In the setting of Lemma~\ref{lem:dual-freeness}, fix a $D$-basis of $E$
and let $\rho:E\to M_r(D)$ be the corresponding regular representation.
Write $d_D=d_{D/L}$ and $d_E=d_{E/L}$. The rule
\[
T\!\left(\sum_\nu b_\nu\,d_E x_\nu\right)
=\sum_\nu\tr\bigl(\rho(b_\nu)d_D\rho(x_\nu)\bigr)
\]
defines a $D$-linear map $T:\Omega_{E/L}\to\Omega_{D/L}$ satisfying
\[
T(c\omega)=0\quad(\omega\in\Omega_{E/L}),
\qquad T(d_Ec)=r\,d_D\varepsilon.
\]
\end{lemma}
\end{greenrevision}
\begin{proof}
\begin{greenrevision}
\smallskip
\noindent\emph{Step 1: The matrix of multiplication by an element.}
Take the $D$-basis $u_1,\ldots,u_r$ fixed in the statement.
For $x\in E$, commutativity makes the multiplication map $m_x(v)=xv$
$D$-linear, and its matrix $\rho(x)=(r_{ij}(x))$ is determined by
\[
xu_j=\sum_{i=1}^r r_{ij}(x)u_i\qquad(1\le j\le r).
\]
Addition and composition of multiplication maps give
\[
\rho(x+y)=\rho(x)+\rho(y),\qquad
\rho(xy)=\rho(x)\rho(y),\qquad \rho(1)=I_r.
\]
For $\delta\in D$, scalar multiplication has matrix $\rho(\delta)=\delta I_r$,
where $D$ is identified with its image in $E$.
In particular, $\varepsilon$ acts as multiplication by $c$, so
\[
\rho(c)=\varepsilon I_r.
\]

\smallskip
\noindent\emph{Step 2: Differentials in the dual-number algebra.}
Write $d_D=d_{D/L}$ and $d_E=d_{E/L}$ to distinguish the
two universal derivations. Since $D=L[\varepsilon]/(\varepsilon^2)$,
the presentation of differentials of a quotient gives
\[
\Omega_{D/L}
=\frac{D\,d_D\varepsilon}{D(2\varepsilon\,d_D\varepsilon)}
\simeq L\,d_D\varepsilon.
\]
Indeed, the polynomial algebra has differential module freely
generated by $d\varepsilon$, and the relation $\varepsilon^2=0$
imposes $d(\varepsilon^2)=2\varepsilon\,d\varepsilon=0$.
As $2$ is invertible in $L$, the quotient is
$(D/(\varepsilon))\,d_D\varepsilon$.
It follows in particular that
\[
d_D\varepsilon\ne0,\qquad
\varepsilon\,d_D\varepsilon=0.
\]
For $\alpha,\beta\in L$, differentiation and the module action
therefore have the concrete forms
\[
d_D(\alpha+\beta\varepsilon)=\beta\,d_D\varepsilon,\qquad
(\alpha+\beta\varepsilon)d_D\varepsilon=\alpha\,d_D\varepsilon.
\]

\smallskip
\noindent\emph{Step 3: A concrete formula for the trace pairing.}
Decompose each matrix entry into its constant and $\varepsilon$ terms, writing
\[
\rho(x)=P(x)+\varepsilon Q(x),\qquad P(x),Q(x)\in M_r(L).
\]
Entrywise differentiation gives $d_D\rho(x)=Q(x)\,d_D\varepsilon$.
Since $\varepsilon\,d_D\varepsilon=0$, for $b,x\in E$ we have
\[
\Phi(b,x):=\tr\bigl(\rho(b)d_D\rho(x)\bigr)
=\tr\bigl(P(b)Q(x)\bigr)d_D\varepsilon.
\]
Here the trace of a matrix with entries in $\Omega_{D/L}$ is again the sum of its diagonal entries.
The proposed map is thus written as
\[
T\!\left(\sum_\nu b_\nu\,d_E x_\nu\right)
=\sum_\nu\Phi(b_\nu,x_\nu).
\]
Every differential form admits such a finite expression, but the expression need not be unique.
We next check that the result is independent of the expression chosen.

\smallskip
\noindent\emph{Step 4: Why this rule is well-defined.}
The module $\Omega_{E/L}$ is generated by $d_E x$,
subject to additivity in $x$, vanishing on $L$, and the
Leibniz relation. To define an additive map on this module
by a rule for $b\,d_E x$, we must respect additivity in $b$
and these relations with every coefficient $b\in E$.
In particular, the required Leibniz identity is
\[
\Phi(b,xy)=\Phi(by,x)+\Phi(bx,y).
\]

The maps $P,Q:E\to M_r(L)$ are $L$-linear, so $\Phi$
is additive in both arguments.
For $\lambda\in L$, one has
$\rho(\lambda)=\lambda I_r$, whence $Q(\lambda)=0$
and $\Phi(b,\lambda)=0$.
Comparing constant and $\varepsilon$ coefficients in
$\rho(xy)=\rho(x)\rho(y)$ gives
\[
\begin{aligned}
P(xy)&=P(x)P(y),\\
Q(xy)&=Q(x)P(y)+P(x)Q(y).
\end{aligned}
\]
It follows that
\[
\begin{aligned}
\Phi(b,xy)
&=\bigl\{\tr(P(b)Q(x)P(y))
       +\tr(P(b)P(x)Q(y))\bigr\}d_D\varepsilon\\
&=\bigl\{\tr(P(y)P(b)Q(x))
       +\tr(P(bx)Q(y))\bigr\}d_D\varepsilon\\
&=\bigl\{\tr(P(by)Q(x))
       +\tr(P(bx)Q(y))\bigr\}d_D\varepsilon\\
&=\Phi(by,x)+\Phi(bx,y).
\end{aligned}
\]
The first trace in the second line is obtained by cyclically
moving the last factor to the front. For ordinary matrices
$U,V$ over $L$,
\[
\tr(UV)=\sum_{i,j}U_{ij}V_{ji}
       =\sum_{i,j}V_{ji}U_{ij}=\tr(VU),
\]
which proves the cyclic identity by taking
$U=P(b)Q(x)$ and $V=P(y)$.
We also used $P(y)P(b)=P(yb)=P(by)$.
No commutation of $Q(x)$ with $P(y)$ is required.

These checks cover the defining relations with arbitrary
coefficients. Equivalently, start with the free $E$-module
on symbols $e_x$, $x\in E$, send $b e_x$ to $\Phi(b,x)$
additively, and quotient by the $E$-submodule generated by
\[
e_{x+y}-e_x-e_y,\qquad e_\lambda\ (\lambda\in L),\qquad
e_{xy}-y e_x-x e_y.
\]
Every $E$-multiple of these relations is killed by the
identities just verified. This quotient is $\Omega_{E/L}$,
so we have a well-defined additive map
\[
T:\Omega_{E/L}\longrightarrow\Omega_{D/L}.
\]

\smallskip
\noindent\emph{Step 5: Linearity over $D$ and the value at $dc$.}
Let $\delta=\alpha+\beta\varepsilon\in D$, acting on $E$
as $\alpha+\beta c$. Since $\rho(\delta)=\delta I_r$,
comparison of constant terms in $\rho(\delta b)=\delta\rho(b)$
gives $P(\delta b)=\alpha P(b)$.
Therefore
\[
\begin{aligned}
T(\delta b\,d_E x)
&=\tr(P(\delta b)Q(x))d_D\varepsilon\\
&=\alpha\tr(P(b)Q(x))d_D\varepsilon\\
&=\delta\,T(b\,d_E x).
\end{aligned}
\]
The last equality uses $\varepsilon\,d_D\varepsilon=0$.
By additivity, $T$ is $D$-linear on all of $\Omega_{E/L}$.
In particular, because $c$ is the image of $\varepsilon$,
\[
T(c\omega)=\varepsilon T(\omega)=0
\qquad\text{for every }\omega\in\Omega_{E/L}.
\]

On the other hand, $\rho(1)=I_r$ and
$\rho(c)=\varepsilon I_r$ say that $P(1)=I_r$ and $Q(c)=I_r$.
Thus
\begin{equation}\label{eq:trace-c}
T(d_E c)
=T(1\cdot d_E c)
=\tr(P(1)Q(c))d_D\varepsilon
=r\,d_D\varepsilon\ne0.
\end{equation}
The last inequality holds because $r>0$,
$\operatorname{char}L=0$, and $d_D\varepsilon\ne0$. \qedhere
\end{greenrevision}
\end{proof}

\begin{lemma}\label{lem:finite-trace}
Let \(L\) be a field of characteristic zero and let \(E\) be a
finite-dimensional commutative \(L\)-algebra. If \(0\ne c\in E\)
satisfies \(c^2=0\) and \(\Ann_E(c)=cE\), then
\[
d_{E/L}c\notin c\Omega_{E/L}.
\]
\end{lemma}
\begin{proof}
\begin{greenrevision}
Put $D=L[\varepsilon]/(\varepsilon^2)$, acting on $E$ via $\varepsilon\mapsto c$,
and write $r=\dim_L(cE)>0$. Lemma~\ref{lem:dual-freeness} gives $E\simeq D^r$.
By Lemma~\ref{lem:differential-trace-map} and \eqref{eq:trace-c}, there is a $D$-linear map $T$ with
\[
T(c\Omega_{E/L})=0,\qquad T(d_Ec)=r\,d_D\varepsilon\ne0.
\]
If $d_Ec$ belonged to $c\Omega_{E/L}$, applying $T$ would give both zero and a nonzero value,
a contradiction. This proves the lemma. \qedhere

\end{greenrevision}
\end{proof}

\begin{greenrevision}
\begin{example}[The trace calculation in a two-by-two matrix]
Take $E=L[t]/(t^4)$ and $c=t^2$.
Retain the notation $d_D=d_{D/L}$ and $d_E=d_{E/L}$
for the two universal derivations.
Then $\Ann_E(c)=cE$, and $E$ is free of rank two over
$D=L[\varepsilon]/(\varepsilon^2)$ via $\varepsilon\mapsto t^2$,
with basis $(1,t)$.
Multiplication by $t$ sends the first basis vector to $t$
and the second to $t^2=\varepsilon\cdot1$. Thus
\[
\rho(t)=
\begin{pmatrix}0&\varepsilon\\1&0\end{pmatrix},
\qquad
\rho(c)=\rho(t)^2=\varepsilon I_2,
\qquad
d_D\rho(t)=
\begin{pmatrix}0&d_D\varepsilon\\0&0\end{pmatrix}.
\]
In this example the trace map gives
\[
\begin{aligned}
T(d_Et)
&=\tr\bigl(d_D\rho(t)\bigr)=0,\\
T(t\,d_Et)
&=\tr\!\left(
\begin{pmatrix}0&\varepsilon\\1&0\end{pmatrix}
\begin{pmatrix}0&d_D\varepsilon\\0&0\end{pmatrix}
\right)=d_D\varepsilon.
\end{aligned}
\]
Consequently
$T(d_Ec)=T(2t\,d_Et)=2\,d_D\varepsilon\ne0$.
This calculation illustrates why the definition keeps the
coefficient $b$ inside $\rho(b)$ in $T(b\,d_E x)$.
In particular, $T(d_Et)=0$ does not imply $T(t\,d_Et)=0$;
the linearity established in the proof is over $D$.
\end{example}
\end{greenrevision}

\begin{remark}
\rev{The explicit basis construction is equivalent to saying that
the Jordan form of $m_c$ consists entirely of size-two blocks.
Once a basis is fixed, the differential obstruction is detected
by the trace of the regular representation. No basis-independence
of the auxiliary map $T$ is needed, and neither a Gorenstein nor
a complete-intersection hypothesis is required.}
\end{remark}

\begin{greenrevision}
The preceding finite-dimensional obstruction, combined with Lemma~\ref{lem:colon-exact},
rules out the hypothetical containment. More precisely, let $A$ be a nonzero finite-dimensional
commutative $k$-algebra, with $\operatorname{char}k=0$, and let $a\in A$ be nilpotent
with $d_{A/k}a=0$. If $a=0$, then $\Ann_A(a)=A\ne0$ whereas $aA=0$,
so the required noncontainment is immediate. If $a\ne0$, suppose that $\Ann_A(a)\subseteq aA$.
By Lemma~\ref{lem:colon-exact}, in the finite-dimensional quotient $C=A/(a^2)$,
the element $c=a\bmod(a^2)$ satisfies $c\ne0$, $c^2=0$, and $\Ann_C(c)=cC$.
The natural differential map
\[
\Omega_{A/k}\otimes_A C\longrightarrow\Omega_{C/k},\qquad
(d_{A/k}x)\otimes1\longmapsto d_{C/k}\bar x
\]
sends $d_{A/k}a=0$ to $d_{C/k}c=0$.
On the other hand, Lemma~\ref{lem:finite-trace}, with $L=k$ and $E=C$,
gives $d_{C/k}c\notin c\Omega_{C/k}$, a contradiction.
Hence $\Ann_A(a)\nsubseteq aA$.

For the Jacobian quotient at the beginning of this section, vanishing of the differential
is ensured by Lemma~\ref{lem:finite-jet}. Together with Proposition~\ref{prop:quotient-equivalence},
this gives $J_f:f\nsubseteq(J_f,f)$.
The next section extends the preceding finite-dimensional differential obstruction to Noetherian algebras.
\end{greenrevision}

\section{Conormal rigidity over Noetherian algebras}
\label{sec:noether-rigidity}

This section reduces the Noetherian case to the finite-dimensional one.
Localization at a minimal prime produces an Artinian local ring, and a
copy of its residue field inside the ring then places it within the
scope of Lemma~\ref{lem:finite-trace}. The following coefficient-field
lemma supplies the second step.
\begin{revision}
A \emph{coefficient field} of a local ring $(E,\mathfrak m,L)$
is a subfield mapping isomorphically to the residue field $L$.
We need such a field to contain the given copy of $k$, not merely
an abstract coefficient field. The next proof gives this relative
version directly; compare the Cohen structure theorem
\cite[Tag 0323]{Stacks}.
\end{revision}

\begin{lemma}\label{lem:coefficient-field}
Let \(k\) be a field of characteristic zero and let \((E,\mathfrak m,L)\)
be an Artinian local \(k\)-algebra with residue field
\(L=E/\mathfrak m\). Then the residue map \(E\twoheadrightarrow L\)
admits a \(k\)-algebra section.
\end{lemma}
\begin{proof}
\begin{revision}
Choose a transcendence basis $\mathcal T$ of $L/k$ and lift each
$t\in\mathcal T$ to $\widetilde t\in E$. Evaluation at these lifts
embeds $k[\mathcal T]$ in $E$: a nonzero polynomial has nonzero
residue because $\mathcal T$ is algebraically independent.
Every such value is a unit in $E$. Hence the embedding extends
to $F=k(\mathcal T)\hookrightarrow E$ and lifts the given inclusion
$F\subseteq L$.

The extension $L/F$ is algebraic and separable. Let $F'\subseteq L$
be finite over $F$. By the primitive element theorem,
$F'=F(\theta)$ for a separable element $\theta$.
Let $P(T)\in F[T]$ be its monic minimal polynomial and
choose $u_0\in E$ reducing to $\theta$. Then
$P(u_0)\in\mathfrak m$ and $P'(u_0)$ is a unit.
The iteration
\[
u_{j+1}=u_j-P(u_j)P'(u_j)^{-1}
\]
satisfies $P(u_j)\in\mathfrak m^{2^j}$, as follows by expanding
$P(u_j+v)$ to first order and noting that the remaining terms
are divisible by $v^2$. The residue of every $u_j$ is $\theta$,
so all the inverses in this formula exist.
Since $\mathfrak m$ is nilpotent, a finite number of iterations
gives a root $u$ of $P$ lifting $\theta$.
Thus $F[T]/(P)\to E$, $T\mapsto u$, is an embedding lifting
$F'\subseteq L$.

This lift is unique over the chosen copy of $F$.
Indeed, two roots $u,v$ of $P$ with residue $\theta$ satisfy
$0=P(u)-P(v)=(u-v)U$, where $U$ reduces to
$P'(\theta)\ne0$, so $u=v$.
The same argument applied to the minimal polynomial of every
element shows that the embeddings of finite intermediate fields
are compatible under inclusion. Taking their union yields a
$k$-algebra section $L\hookrightarrow E$ of the residue map.
This also covers transcendence bases and algebraic extensions
of arbitrary cardinality. \qedhere
\end{revision}
\end{proof}

Once the coefficient field has been chosen, the localized Artinian ring
can be regarded as a finite-dimensional algebra over its residue field.
We can now lift the finite-dimensional obstruction of the preceding
section to Noetherian algebras.

\begin{theorem}\label{thm:conormal-rigidity}
Let \(k\) be a field of characteristic zero and let \(C\) be a
Noetherian commutative \(k\)-algebra. If \(0\ne c\in C\) satisfies
\[
c^2=0,\qquad\Ann_C(c)=cC,
\]
then
\[
\greenrev{dc\notin c\Omega_{C/k}.}
\]
In particular, \(dc\ne0\).
\end{theorem}
\begin{proof}
\begin{greenrevision}
Suppose $d_{C/k}c=c\omega$. Choose a minimal prime
$\mathfrak p$ of $C$ and put $E=C_{\mathfrak p}$.
Prime ideals of $E$ correspond to prime ideals of $C$
contained in $\mathfrak p$. Minimality therefore makes
$\mathfrak pE$ the only prime ideal of $E$.
Thus $E$ is zero-dimensional and Noetherian, hence Artinian
local.

Write $c_E=c/1$. This element remains nonzero: if $c_E=0$,
there would be $s\notin\mathfrak p$ with $sc=0$, whereas
\[
s\in\Ann_C(c)=cC\subseteq\mathfrak p.
\]
The last inclusion holds because $c^2=0$, so every prime
contains $c$. Clearly $c_E^2=0$. To check the annihilator
after localization, suppose $(x/s)c_E=0$.
Then $txc=0$ for some $t\notin\mathfrak p$, so
$tx\in cC$ and $x/s=(tx)/(ts)\in c_EE$.
The reverse inclusion follows from $c_E^2=0$. Hence
\[
c_E\ne0,\qquad c_E^2=0,\qquad\Ann_E(c_E)=c_EE.
\]

Let $L=E/\mathfrak m_E$. Lemma~\ref{lem:coefficient-field}
gives a section $L\hookrightarrow E$ over $k$.
There exists $N$ with $\mathfrak m_E^N=0$.
Each quotient $\mathfrak m_E^i/\mathfrak m_E^{i+1}$ is a
finite $E$-module annihilated by $\mathfrak m_E$, hence a
finite-dimensional $L$-vector space. The coefficient-field
section makes the entire filtration a filtration of
$L$-vector spaces, and therefore
\[
\dim_L E
=\sum_{i=0}^{N-1}
 \dim_L(\mathfrak m_E^i/\mathfrak m_E^{i+1})<\infty.
\]
In particular, $E$ satisfies the finiteness hypothesis of
Lemma~\ref{lem:finite-trace}, even if $L/k$ has infinite degree.

Differentials commute with localization:
\[
\Omega_{C/k}\otimes_C E\simeq\Omega_{E/k}.
\]
Thus the assumed equality becomes
$d_{E/k}c_E=c_E\omega_E$.
The coefficient field contains the given copy of $k$, so
the transitivity sequence supplies a surjection
\[
\Omega_{E/k}\longrightarrow\Omega_{E/L},
\qquad d_{E/k}x\longmapsto d_{E/L}x.
\]
It sends this equality to $d_{E/L}c_E=c_E\omega'$,
contradicting Lemma~\ref{lem:finite-trace}.

Therefore the assumption $d_{C/k}c\in c\Omega_{C/k}$ is false, and hence
\[
d_{C/k}c\notin c\Omega_{C/k}.
\]
Since $0\in c\Omega_{C/k}$, in particular $d_{C/k}c\ne0$, which proves the theorem. \qedhere
\end{greenrevision}
\end{proof}

\begin{revision}
\begin{remark}\label{rem:conormal-image}
For $H=cC$ as in Theorem~\ref{thm:conormal-rigidity}, the
conormal map sends the generator $c\in H/H^2$ to the nonzero
class of $dc$ in $\Omega_{C/k}/c\Omega_{C/k}$.
The theorem asserts nonvanishing of this class, not injectivity
of the entire conormal map.
\end{remark}
\end{revision}

\section{Ideal quotient theorems}\label{sec:colon-theorems}

Theorem~\ref{thm:conormal-rigidity} concerns the differential of a
self-exact square-zero element. We now translate it back into \greenrev{an}
ideal quotient statement. If the desired noncontainment were to fail,
Lemma~\ref{lem:colon-exact} would produce such a square-zero element in
a quotient ring, while differential vanishing would contradict
conormal rigidity. This proves the first main theorem.

\begin{theorem}\label{thm:noetherian-colon}
Let \(k\) be a field of characteristic zero, let \(R\) be a commutative
\(k\)-algebra, let \(I\subsetneq R\), and assume that \(A=R/I\) is
Noetherian. For \(f\in R\), write \(a=[f]\in A\). If \(a\) is
nilpotent and \(d_{A/k}a=0\), then
\[
I:f\not\subseteq(I,f).
\]
\end{theorem}
\begin{proof}
If \(a=0\), then \(I:f=R\), while \((I,f)=I\ne R\). Assume that
\(a\ne0\) and, toward a contradiction, that
\(I:f\subseteq(I,f)\). Modulo \(I\), this is
\(\Ann_A(a)\subseteq aA\). Lemma~\ref{lem:colon-exact} produces, in
\(C=A/(a^2)\), an element \(c=a\bmod(a^2)\) satisfying
\[
c\ne0,\qquad c^2=0,\qquad\Ann_C(c)=cC.
\]
\begin{greenrevision}
The quotient map $A\to C$ induces the natural surjection
\[
\Omega_{A/k}\otimes_A C\longrightarrow\Omega_{C/k},
\qquad da\otimes1\longmapsto dc.
\]
The hypothesis $da=0$ therefore implies $dc=0$.
The ring $C$ is Noetherian because it is a quotient of $A$,
and $c\ne0$ ensures that $C$ is nonzero. Thus all the
hypotheses of Theorem~\ref{thm:conormal-rigidity} apply,
giving the required contradiction. \qedhere
\end{greenrevision}
\end{proof}

The preceding theorem produces an element of the ideal quotient outside
\((I,f)\). Over a local ring, we now show that this element can be chosen
with nonzero class modulo the maximal ideal times the ideal quotient.
Nakayama's lemma gives the following strengthening.

\begin{theorem}\label{thm:local-strengthening}
Under the assumptions of Theorem~\ref{thm:noetherian-colon}, suppose
that \((R,\mathfrak n)\) is local and put \(K=I:f\). Then
\[
\greenrev{K\not\subseteq(I,f)+\mathfrak nK.}
\]
\end{theorem}
\begin{proof}
\begin{greenrevision}
Let $\mathfrak m_A=\mathfrak n/I$ be the maximal ideal of
$A=R/I$, and put
\[
M=K/I=\Ann_A(a),\qquad
N=(M+aA)/aA\simeq M/(M\cap aA).
\]
Theorem~\ref{thm:noetherian-colon} gives $N\ne0$.
Since $A$ is Noetherian, its ideal $M$ is finitely generated;
the same is true of its quotient $N$.

The natural map $K\to N$ sends $g$ to
$(g+I)+aA$. It is surjective and has kernel $K\cap(I,f)$.
Under this map the image of $\mathfrak nK$ is exactly
$\mathfrak m_A N$: multiplication by $r\in\mathfrak n$
on a lift in $K$ induces multiplication by $r+I$ on $N$.
Consequently
\[
N/\mathfrak m_A N
\simeq
\frac{K}{(K\cap(I,f))+\mathfrak nK}.
\]
Nakayama's lemma \cite[Proposition 2.8]{AM} shows that
$N/\mathfrak m_A N\ne0$, since $N$ is nonzero and finite.
Finally, $\mathfrak nK\subseteq K$ implies
\[
K\cap\bigl((I,f)+\mathfrak nK\bigr)
=(K\cap(I,f))+\mathfrak nK.
\]
The nonzero quotient therefore yields an element of $K$
outside $(I,f)+\mathfrak nK$, as required.
Only $M$ and $N$ are required to be finite over $A$;
the argument does not assume that $K$ is finite over $R$. \qedhere
\end{greenrevision}
\end{proof}

\begin{revision}
\begin{remark}\label{rem:minimal-generator}
Writing $K=I:f$ and $\kappa=R/\mathfrak n$, the local theorem
is equivalent to the nonvanishing of the $\kappa$-vector space
\[
\frac{K}{(K\cap(I,f))+\mathfrak nK}.
\]
Indeed, $K\cap((I,f)+\mathfrak nK)
=(K\cap(I,f))+\mathfrak nK$ because $\mathfrak nK\subseteq K$.
An element $g$ with nonzero image in this quotient has nonzero
class in $K/\mathfrak nK$. If $K$ is finitely generated, in
particular if $R$ is Noetherian, extend that class to a
$\kappa$-basis and lift the basis; Nakayama's lemma shows that
the lifts form a minimal generating set containing $g$.
When only $R/I$ is assumed Noetherian, $K$ itself need not be
finitely generated, so the vector-space formulation is the
appropriate statement.
\end{remark}

\begin{example}[The characteristic-zero hypothesis is essential]
Let $\operatorname{char}k=p>0$, set $A=k[t]/(t^{2p})$, and take
$a=t^p$. Then $a\ne0$, $a^2=0$, and $da=pt^{p-1}dt=0$, but
$\Ann_A(a)=aA$. In fact, $at^i=0$ exactly when $i\geq p$.
Thus the abstract annihilator theorem cannot be extended
unchanged to positive characteristic.
\end{example}
\end{revision}

\section{Jacobian ideals and critical loci}\label{sec:applications}

We now return to the original Jacobian-ideal problem from
Section~\ref{sec:background}. To apply
Theorem~\ref{thm:local-strengthening} to
\[
A=k[[x_1,\ldots,x_n]]/J_f,\qquad a=[f],
\]
we must verify that \(A\) is Noetherian, that \(a\) is nilpotent, and
that \(d_{A/k}a=0\). When \(J_f\) is \(\mathfrak m\)-primary, the first
two statements follow because \(A\) is Artinian local. Differential
vanishing requires a separate argument.

One cannot, without further justification, write the differential of a
formal power series as \(df=\sum_i(\partial_i f)\,dx_i\). We use ordinary
K\"ahler differentials, and ordinary derivations need not be continuous
for the \(\mathfrak m\)-adic topology, so the infinite-series formula is
not directly available. Instead, modulo \(J_f\), we replace \(f\) by a
finite polynomial truncation, for which the usual finite
differentiation formula applies. The next lemma carries this out.

\begin{lemma}\label{lem:finite-jet}
Let \(R=k[[x_1,\ldots,x_n]]\), assume that \(J_f\) is
\(\mathfrak m\)-primary, and set \(A=R/J_f\) and \(a=[f]\in A\).
Then \(d_{A/k}a=0\).
\end{lemma}
\begin{proof}
\greenrev{Choose $N$ with $\mathfrak m^N\subseteq J_f$.
Let $P$ be the sum of all terms of $f$ of total degree at
most $N$. Formal differentiation lowers the order of the
discarded tail by at most one. Thus}
\[
f-P\in\mathfrak m^{N+1}\subseteq J_f,\qquad
\partial_i f-\partial_iP\in\mathfrak m^N\subseteq J_f.
\]
Since \(\partial_i f\in J_f\), every \(\partial_iP\) vanishes in \(A\).
The elements \(P\) and \(f\) represent the same class in \(A\), so
\[
d_{A/k}a=d_{A/k}[P]
=\sum_i[\partial_iP]\,d[x_i]=0.
\qedhere
\]
\end{proof}

Thus the abstract local ideal quotient theorem applies to the original formal
power-series setting and gives the required conclusion.

\begin{corollary}\label{cor:hns}
Let \(k\) be a field of characteristic zero and let
\(R=k[[x_1,\ldots,x_n]]\). \rev{If $f\in\mathfrak m^2$ and $J_f$ is $\mathfrak m$-primary, then}
\[
\greenrev{J_f:f\not\subseteq
(J_f,f)+\mathfrak m(J_f:f).}
\]
In particular, \(J_f:f\not\subseteq(J_f,f)\).
\end{corollary}
\begin{proof}
The algebra \(A=R/J_f\) is Artinian local, \rev{and
$[f]$ belongs to its maximal ideal because $f(0)=0$.}
Thus \([f]\) is nilpotent.
Apply Lemma~\ref{lem:finite-jet} and
Theorem~\ref{thm:local-strengthening}. \qedhere
\end{proof}

\begin{revision}
\begin{corollary}[Numerical consequences]\label{cor:numerical}
Under the hypotheses of Corollary~\ref{cor:hns}, set
$e=e^{\mathrm{BS}}(f)$. Then $\beta_f=0$ and
\[
\dim_k R/(J_f,f^2)<2\tau_f,\qquad
\mu_f\leq e(\tau_f-1)+1.
\]
If $f\notin J_f$, so $e\geq2$, then
\[
1<\frac{\mu_f}{\tau_f}<e\leq n.
\]
\end{corollary}
\begin{proof}
\begin{greenrevision}
Corollary~\ref{cor:hns} gives $\Ann_A(a)\nsubseteq aA$
for $A=R/J_f$ and $a=[f]$. Thus $m_a$ has at least one
size-one Jordan block by \eqref{eq:jordan-count}.
Let $b_s$ denote the number of its blocks of size $s$.
The equality $a^e=0$ means that every block has size at
most $e$, and hence
\[
\mu_f=\sum_{s=1}^{e}s b_s,\qquad
\tau_f=\sum_{s=1}^{e}b_s,\qquad b_1\geq1.
\]
Here $\tau_f=\dim_k(A/aA)$, so the second identity follows
because every block contributes one to that quotient.
Similarly, $R/(J_f,f^2)\simeq A/a^2A$ gives
\[
\dim_kR/(J_f,f^2)
=b_1+2\sum_{s=2}^{e}b_s
=2\tau_f-b_1<2\tau_f.
\]
For the Milnor number, separate the size-one blocks:
\[
\begin{aligned}
\mu_f
&=b_1+\sum_{s=2}^{e}s b_s\\
&\leq b_1+e\sum_{s=2}^{e}b_s
=e\tau_f-(e-1)b_1
\leq e(\tau_f-1)+1.
\end{aligned}
\]
These formulas also cover $e=1$, when all blocks have size
one and $\mu_f=\tau_f$.

To see $\beta_f=0$ directly, note that $J_f\subseteq J_f:f$.
If $(J_f:f)+(f)$ equalled $I_f$, then
$J_f:f\subseteq I_f$, contradicting Corollary~\ref{cor:hns}.
The ascending chain of colon ideals rules out equality for
every larger exponent as well, while exponent zero always
gives $I_f$.

If $f\notin J_f$, multiplication by $a$ is nonzero and at
least one Jordan block has size at least two. Thus
$\mu_f>\tau_f$ and $e\geq2$.
The preceding estimate gives
$\mu_f\leq e\tau_f-(e-1)<e\tau_f$.
Dividing by the positive integer $\tau_f$ and applying
$e\leq n$ proves the last assertion. The Milnor-number
bound agrees with the consequence predicted in
\cite[Proposition 3.7]{HNS}. \qedhere
\end{greenrevision}
\end{proof}

\begin{remark}
The hypothesis on the critical value is necessary. For
$f=1+x_1^2+\cdots+x_n^2$, the ideal $J_f=\mathfrak m$ is
$\mathfrak m$-primary, but $[f]=1$ in $R/J_f$; here
$J_f:f=J_f$ and $(J_f,f)=R$. The finite-jet lemma still gives
$d[f]=0$, but nilpotence fails.
\end{remark}
\end{revision}

The formal power-series case uses a finite truncation to avoid the
continuity issue. For a smooth finite-type algebra, the cotangent module
is finite projective, and a dual basis expresses \(df\) directly as a
linear combination of the derivatives \(\delta(f)\). This gives the
corresponding local statement.

\begin{corollary}\label{cor:smooth-local}
Let \((R,\mathfrak n)\) be a localization of a smooth finitely generated
\(k\)-algebra, where \(\operatorname{char}k=0\), and define
\[
J_f=(\delta(f):\delta\in\operatorname{Der}_k(R)).
\]
If \(R/J_f\) is Artinian local and \(f\in\mathfrak n\), then
\[
J_f:f\not\subseteq(J_f,f)+\mathfrak n(J_f:f).
\]
\end{corollary}
\begin{proof}
Smoothness makes \(\Omega_{R/k}\) finite projective. Choose a dual
basis \(\omega_1,\ldots,\omega_s\in\Omega_{R/k}\) and
\(\varphi_1,\ldots,\varphi_s\in
\operatorname{Hom}_R(\Omega_{R/k},R)\), so that
\[
\omega=\sum_j\varphi_j(\omega)\omega_j
\quad\text{for every }\omega\in\Omega_{R/k}.
\]
By the universal property of K\"ahler differentials, every
\(\varphi_j\) corresponds to a derivation
\(\delta_j\in\operatorname{Der}_k(R)\). Applying the dual-basis
identity to \(df\) gives
\[
df=\sum_j\delta_j(f)\omega_j\in J_f\Omega_{R/k},
\]
hence
\(d_{R/J_f}[f]=0\). The remaining assumptions imply that \([f]\) is
nilpotent, so Theorem~\ref{thm:local-strengthening} applies. \qedhere
\end{proof}

\begin{revision}
\begin{remark}[Other power-series categories]\label{rem:other-categories}
\begin{greenrevision}
The finite-jet argument also applies to the following three classes of rings, whose maximal ideals are denoted by $\mathfrak m=(x_1,\ldots,x_n)$:
\begin{enumerate}
\item the polynomial local ring $k[x_1,\ldots,x_n]_{(x_1,\ldots,x_n)}$;
\item the algebraic power-series ring $k\langle x_1,\ldots,x_n\rangle$, consisting of the series in $k[[x_1,\ldots,x_n]]$ algebraic over $k(x_1,\ldots,x_n)$;
\item the convergent power-series ring $k\{x_1,\ldots,x_n\}$ when $k=\mathbb R$ or $\mathbb C$.
\end{enumerate}

In each case, let $f\in\mathfrak m^2$ and suppose $J_f$ is $\mathfrak m$-primary. Choose $N$ with $\mathfrak m^N\subseteq J_f$. Taylor truncation gives a polynomial $P$ satisfying
\[
f-P\in\mathfrak m^{N+1},\qquad
\partial_i f-\partial_i P\in\mathfrak m^N.
\]
Thus the proof of Lemma~\ref{lem:finite-jet} applies verbatim and gives $d_{R/J_f\,/k}[f]=0$ for the ordinary differentials of the finite critical quotient $R/J_f$, yielding the corresponding isolated-critical-point statement.

These ideal memberships may also be checked after the faithfully flat completion $R\to k[[x_1,\ldots,x_n]]$ and then contracted to $R$; see \cite[Chapter 10]{AM}.

Furthermore, a complete equicharacteristic regular local ring with a chosen characteristic-zero coefficient field is a formal power-series ring over that field by the Cohen structure theorem \cite[Tag 0323]{Stacks}, so the preceding statement applies there as well. This argument still requires the critical quotient to have finite length and does not give a statement about arbitrary non-isolated formal critical quotients.
\end{greenrevision}
\end{remark}
\end{revision}

The preceding two corollaries assume that the critical quotient is
Artinian, so nilpotence of \([f]\) follows from finite length. When the
critical locus is positive-dimensional, this implication is no longer
available. We first record a fact about differential constants in
function fields.

\begin{lemma}\label{lem:differential-constants}
Let \(K/k\) be a finitely generated field extension of characteristic
zero. If \(u\in K\) satisfies \(d_{K/k}u=0\), then \(u\) is algebraic
over \(k\).
\end{lemma}
\begin{proof}
Choose a separating transcendence basis
\(t_1,\ldots,t_d\) for \(K/k\). Then \(K\) is finite separable over
\(k(t_1,\ldots,t_d)\), and
\[
\Omega_{K/k}\simeq
\bigoplus_{i=1}^dK\,dt_i.
\]
If \(u\) were transcendental over \(k\), it could be included in a
transcendence basis after replacing the chosen basis. In the
corresponding decomposition, \(du\) would be a basis vector and could
not vanish. Thus \(u\) is algebraic over \(k\). \qedhere
\end{proof}

The lemma allows us to work component by component. Differential
vanishing makes the function algebraic over the base field on each
irreducible component; after subtracting its value at the chosen
rational critical point, that algebraic constant must be zero. This
recovers the nilpotence needed for the main theorem.

\begin{corollary}\label{cor:positive-dimensional}
Let \(R\) be a smooth finitely generated algebra over a field \(k\) of
characteristic zero, let \(f\in R\), and let \(x\) be a \(k\)-rational
critical point. \greenrev{Thus $R/\mathfrak m_x=k$ and
$J_f\subseteq\mathfrak m_x$, and $f(x)$ denotes the residue
of $f$ in $k$.} The critical locus need not be zero-dimensional. Put
\[
B=(R/J_f)_{\mathfrak m_x},\qquad a=[f-f(x)]\in B.
\]
Then \(a\) is nilpotent and
\[
\rev{\begin{aligned}
\bigl(J_f:(f-f(x))\bigr)_{\mathfrak m_x}
&\not\subseteq (J_f,f-f(x))_{\mathfrak m_x}\\
&\quad+\mathfrak m_xR_{\mathfrak m_x}
 \bigl(J_f:(f-f(x))\bigr)_{\mathfrak m_x}.
\end{aligned}}
\]
\end{corollary}
\begin{proof}
As above, \(d_{B/k}a=0\). Let \(\mathfrak p\) be a minimal prime of
\(B\), and let \(K=\operatorname{Frac}(B/\mathfrak p)\).
\greenrev{This field is finitely generated over $k$: if
$\mathfrak q\subseteq R$ is the inverse image of
$\mathfrak p$, then $B/\mathfrak p$ is a localization of
the finite-type domain $R/\mathfrak q$, and both have the
same fraction field.}
The image \(\bar a\in K\) has zero differential, so
Lemma~\ref{lem:differential-constants} shows that \(\bar a\) is
algebraic over \(k\). \rev{Let $P(T)\in k[T]$ be the monic minimal polynomial of
$\bar a$. Since $B/\mathfrak p$ injects into its fraction field,
$P(a\bmod\mathfrak p)=0$ already in $B/\mathfrak p$.
The local ring $B/\mathfrak p$ has residue field $k$ at $x$,
because $\mathfrak p$ is contained in the maximal ideal of $B$
and $x$ is $k$-rational. Applying this residue map, and using
$a(x)=0$, gives $P(0)=0$.} Irreducibility then forces
\(P(T)=T\), so \(\bar a=0\). This holds for every minimal prime, hence
\[
a\in\bigcap_{\mathfrak p\in\operatorname{Min}B}\mathfrak p
=\sqrt{0_B}.
\]
The nilradical of a Noetherian ring is nilpotent, so \(a\) is
nilpotent. \rev{Apply Theorem~\ref{thm:local-strengthening} with
the local ambient ring $R_{\mathfrak m_x}$, the ideal
$J_fR_{\mathfrak m_x}$, and $f-f(x)$. Colon ideals by a single
element commute with localization: a relation holding after
localization can be cleared by one denominator. This identifies
the resulting local colon ideal with the one displayed above.} \qedhere
\end{proof}

\section{The integral-closure strengthening is false}
\label{sec:integral-counterexample}

\rev{We now disprove Conjecture~\ref{conj2}. This is a failure
of enlargement of the target ideal to $\overline{J_f}$, not a
failure of the local strengthening proved above.}

\begin{revision}
For a Noetherian local ring $(S,\mathfrak m)$, its
\emph{associated graded ring} is
\[
\gr_{\mathfrak m}S=\bigoplus_{q\geq0}
\mathfrak m^q/\mathfrak m^{q+1}.
\]
If $0\ne h\in S$, write
$\ord_{\mathfrak m}(h)=\max\{q:h\in\mathfrak m^q\}$ and denote
its image in degree $q$ by $h^*$, its \emph{initial form}.
The order is finite by the Krull intersection theorem.
For a quotient $S/H$ we always use the filtration by powers
of its maximal ideal, namely the image of $\mathfrak m$.
The initial ideal $H^*$ is the homogeneous ideal generated
by the initial forms of all nonzero elements of $H$.
The following standard filtered criterion is a special case
of the results of Valabrega and Valla \cite{VV}.
\end{revision}

\begin{lemma}\label{lem:filtered-regular}
\rev{Let $(S,\mathfrak m)$ be a Noetherian local ring with its
$\mathfrak m$-adic filtration, and let
$g_1,\ldots,g_r\in\mathfrak m$.} If their initial forms
\(g_1^*,\ldots,g_r^*\) form a regular sequence in
\(\gr_{\mathfrak m}S\), then
\[
\gr_{\mathfrak m}\bigl(S/(g_1,\ldots,g_r)\bigr)
\simeq
\gr_{\mathfrak m}S/(g_1^*,\ldots,g_r^*).
\]
\end{lemma}
\begin{proof}
\begin{greenrevision}
For an ideal $H\subseteq S$, the natural graded map
$\gr_{\mathfrak m}S\to\gr_{\mathfrak m}(S/H)$ is
surjective with kernel $H^*$.
Indeed, a degree-$q$ class represented by
$x\in\mathfrak m^q$ lies in the kernel exactly when
$x\in H+\mathfrak m^{q+1}$. Writing $x=h+y$ with
$h\in H$ and $y\in\mathfrak m^{q+1}$ shows that its
nonzero class is the initial form of an element of $H$.
Thus it suffices to identify the initial ideal.

For $r=1$, let $0\ne h\in S$.
The orders of $h$ and $g_1$ are finite by the Krull
intersection theorem. Since $g_1^*$ is a nonzerodivisor,
the product $g_1^*h^*$ is nonzero in the associated graded
ring. Therefore $g_1h\ne0$ and
\[
\ord_{\mathfrak m}(g_1h)
=\ord_{\mathfrak m}(g_1)+\ord_{\mathfrak m}(h),
\qquad (g_1h)^*=g_1^*h^*.
\]
Every nonzero element of $(g_1)$ is such a product.
It follows that $(g_1)^*=(g_1^*)$, proving the assertion
for one element.

For the induction step, put
$S'=S/(g_1,\ldots,g_{r-1})$ with maximal ideal $\mathfrak m'$.
By induction,
\[
\gr_{\mathfrak m'}S'
\simeq
\gr_{\mathfrak m}S/(g_1^*,\ldots,g_{r-1}^*).
\]
If $d=\ord_{\mathfrak m}(g_r)$, regularity of the sequence
of initial forms says that the image of $g_r^*$ in this
quotient is nonzero and is a nonzerodivisor.
The degree-$d$ image of $g_r$ in $S'$ is therefore nonzero.
In particular, its order in $S'$ is still $d$, and its
initial form is precisely the image of $g_r^*$.
The ring $S'$ is Noetherian local, so its maximal-ideal
filtration is separated. Applying the one-element case
in $S'$ completes the induction.
This proves the stated special case of
\cite[Corollary 2.7]{VV}. \qedhere
\end{greenrevision}
\end{proof}

\begin{greenrevision}
We apply the filtered lemma to a concrete five-variable function.
The associated graded algebra determines the length of the
critical quotient and gives a lower bound on the orders of
annihilator elements. The integral-closure identity will be
proved separately, by showing that the Jacobian ideal is
a reduction of $\mathfrak m^3$.
\end{greenrevision}

\begin{revision}
A finite-dimensional graded algebra $G=\bigoplus_{q\geq0}G_q$
generated by $G_1$ has the \emph{strong Lefschetz property} if
there is $\ell\in G_1$ for which
$\times\ell^r:G_q\to G_{q+r}$ has maximal rank for all $q,r\geq0$.
Such an $\ell$ is a strong Lefschetz element.
Monomial complete intersections have this property in
characteristic zero \cite{Stanley,Watanabe}. We need only the
following explicit injective range.

\begin{lemma}\label{lem:lefschetz}
Let $C=\mathbb Q[x_1,\ldots,x_5]/(x_1^3,\ldots,x_5^3)$ and
$\ell=x_1+\cdots+x_5$. If $q,r\geq0$ and $2q+r\leq10$, then
\[
\times\ell^r:C_q\longrightarrow C_{q+r}
\]
is injective. In particular, multiplication by $\ell^5$ is
injective for $q=0,1,2$.
\end{lemma}
\begin{proof}
On $V=\mathbb Q[x]/(x^3)$ define, for $j=0,1,2$,
\[
e(x^j)=x^{j+1},\qquad
h(x^j)=(2j-2)x^j,\qquad
b(x^j)=j(3-j)x^{j-1},
\]
with $e(x^2)=b(1)=0$. Direct evaluation on $1,x,x^2$ gives
\[
[h,e]=2e,\qquad [h,b]=-2b,\qquad[e,b]=h,
\]
where $[X,Y]=XY-YX$ is the commutator.
These are the defining relations of the Lie algebra
$\mathfrak{sl}_2$, the trace-zero $2\times2$ matrices.
On $C\simeq V^{\otimes5}$, sum each of the three operators over
the five factors. Operators on different factors commute,
so the same relations hold. The resulting $e$ is
multiplication by $\ell$, and $h$ acts on $C_q$ by the scalar
$\lambda=2q-10$.

For clarity, the representation-theoretic input is the
complete reducibility of finite-dimensional
$\mathfrak{sl}_2(\mathbb C)$-modules and the description of
their irreducibles \cite[Sections 6--7]{Humphreys}.
We may extend scalars from $\mathbb Q$ to $\mathbb C$, because
the rank of a rational matrix does not change under this
field extension.
An irreducible module of highest weight $m\geq0$ has
one-dimensional weight spaces of weights
$-m,-m+2,\ldots,m$. Choosing vectors $v_0,\ldots,v_m$,
we may normalize the raising operator so that
$e v_j=v_{j+1}$ for $j<m$ and $e v_m=0$, with
$h v_j=(2j-m)v_j$.
\begin{greenrevision}
If this module contains weight $\lambda$, its corresponding
basis vector is $v_j$ with $j=(\lambda+m)/2$, and
$m\geq|\lambda|$.
The normalization gives $e^rv_j=v_{j+r}$ when $j+r\leq m$.
This condition is equivalent to $\lambda+2r\leq m$.
Thus on the one-dimensional $\lambda$-weight space the
map $e^r$ is injective whenever this inequality holds.
\end{greenrevision}
Our assumption gives $\lambda\leq0$ and
\[
\lambda+2r=2q-10+2r\leq10-2q=|\lambda|\leq m.
\]
Consequently $e^r$ is injective on the $\lambda$-weight space
of every irreducible summand, hence on $C_q$.
This proves the assertion over $\mathbb C$ and therefore
over $\mathbb Q$. \qedhere
\end{proof}
\end{revision}

\begin{theorem}\label{thm:integral-counterexample}
Let
\[
R=\mathbb Q[[x_1,\ldots,x_5]],\qquad
\greenrev{\mathfrak m=(x_1,\ldots,x_5)},\qquad
\ell=x_1+\cdots+x_5,
\]
and set
\[
f=\frac14\sum_{i=1}^5x_i^4+\ell^5.
\]
Then \(f\) defines an isolated hypersurface singularity and
\[
\greenrev{J_f:f\subseteq\overline{J_f}.}
\]
More precisely, \(\overline{J_f}=\mathfrak m^3\).
\end{theorem}
\begin{proof}
\begin{greenrevision}
Put $\mathfrak m=(x_1,\ldots,x_5)$ and write
\[
g_i=\partial_i f=x_i^3+5\ell^4,\qquad J=J_f=(g_1,\ldots,g_5).
\]
Their initial forms are $x_1^3,\ldots,x_5^3$.
These form a regular sequence in $\mathbb Q[x_1,\ldots,x_5]$:
after quotienting by the preceding powers, multiplication
by the next power is injective, as can be seen by viewing
the quotient as a polynomial ring in the remaining variables.
Lemma~\ref{lem:filtered-regular} gives
\begin{equation}\label{eq:graded-counterexample}
\gr_{\mathfrak m_A}A
\simeq C:=\mathbb Q[x_1,\ldots,x_5]/(x_1^3,\ldots,x_5^3),
\qquad A=R/J,\quad\mathfrak m_A=\mathfrak m/J.
\end{equation}
The monomials $x_1^{a_1}\cdots x_5^{a_5}$ with
$0\leq a_i\leq2$ form a basis of $C$.
Consequently $\dim_{\mathbb Q}C=3^5=243$, and $C_q=0$
for $q>10$. In particular,
$\mathfrak m_A^{11}/\mathfrak m_A^{12}=0$.
The ideal $\mathfrak m_A^{11}$ is a finite $A$-module, so
Nakayama's lemma applied to
$\mathfrak m_A^{11}=\mathfrak m_A\mathfrak m_A^{11}$
gives $\mathfrak m_A^{11}=0$.
Thus $J$ is $\mathfrak m$-primary, and summing the dimensions
of the finite filtration yields
\[
\mu_f=\dim_{\mathbb Q}A=\sum_{q=0}^{10}\dim_{\mathbb Q}C_q=243.
\]
Since $f\in\mathfrak m^4\mathrel{\textcolor{black}{\subsetneq}}\mathfrak m^2$, the critical
point and the hypersurface singularity are isolated.

We next prove $\overline J=\mathfrak m^3$.
First, $\mathfrak m^3$ is integrally closed.
If $0\ne u$ is integral over it, choose an equation
\[
u^s+b_1u^{s-1}+\cdots+b_s=0,\qquad b_j\in\mathfrak m^{3j}.
\]
Suppose $v=\ord_{\mathfrak m}(u)<3$.
Because $\gr_{\mathfrak m}R$ is a polynomial ring and hence
a domain, the order of a product of nonzero elements is
the sum of their orders. The term $u^s$ has order $sv$,
while every other nonzero term has order at least
$3j+(s-j)v>sv$. Their degree-$sv$ initial forms therefore
cannot cancel. This contradiction gives $v\geq3$.
The element zero also belongs to $\mathfrak m^3$, proving
integral closedness. Since $J\subseteq\mathfrak m^3$,
monotonicity gives $\overline J\subseteq\mathfrak m^3$.

For the reverse inclusion, every monomial of total degree
twelve in five variables is divisible by some $x_i^3$,
because exponents all at most two would have total degree
at most ten. Therefore
\[
\mathfrak m^{12}
=(x_1^3,\ldots,x_5^3)\mathfrak m^9.
\]
Substitute $x_i^3=g_i-5\ell^4$ to obtain
\[
\mathfrak m^{12}\subseteq J\mathfrak m^9+\mathfrak m^{13}.
\]
The reverse containment holds since $J\subseteq\mathfrak m^3$.
For the finite $R$-module $Q=\mathfrak m^{12}/J\mathfrak m^9$,
we consequently have
\[
Q=\frac{J\mathfrak m^9+\mathfrak m^{13}}{J\mathfrak m^9}
=\mathfrak mQ.
\]
Nakayama's lemma yields
\[
\mathfrak m^{12}=J\mathfrak m^9.
\]
Equivalently, $(\mathfrak m^3)^4=J(\mathfrak m^3)^3$,
so $J$ is a reduction of $\mathfrak m^3$.

Here is the integral-dependence argument directly.
Fix $u\in\mathfrak m^3$, put $M=\mathfrak m^9$, and choose
generators $b_1,\ldots,b_t$ of $M$.
Since $uM\subseteq\mathfrak m^{12}=JM$, there are
$c_{ij}\in J$ such that
\[
ub_i=\sum_{j=1}^t c_{ij}b_j\qquad(1\leq i\leq t).
\]
Multiplying this matrix equation by the adjugate of
$uI_t-(c_{ij})$ shows that
$\det(uI_t-(c_{ij}))$ annihilates $M$.
The ring $R$ is a domain and $M\ne0$, so this determinant
vanishes in $R$.
Its expansion has the form
\[
u^t+d_1u^{t-1}+\cdots+d_t=0,\qquad d_j\in J^j.
\]
Thus every $u\in\mathfrak m^3$ is integral over $J$, and
\begin{equation}\label{eq:closure-counterexample}
\overline J=\mathfrak m^3.
\end{equation}
In particular, $f\in\mathfrak m^4\subseteq\overline J$,
so this example satisfies the integrality hypothesis of
Conjecture~\ref{conj2}.

It remains to prove $J:f\subseteq\mathfrak m^3$.
The identities
\[
\sum_i x_ig_i=\sum_i x_i^4+5\ell^5,\qquad
4f=\sum_i x_i^4+4\ell^5
\]
give in $A$
\begin{equation}\label{eq:a-leading}
a=[f]=-\frac14[\ell^5].
\end{equation}
Lemma~\ref{lem:lefschetz} shows that
$\times\ell^5:C_q\to C_{q+5}$ is injective for $q=0,1,2$.
In particular $\ell^5\ne0$ in $C_5$.
Thus $a$ has order exactly five for the
$\mathfrak m_A$-adic filtration, and its initial form is
$a^*=-\frac14\ell^5$ under \eqref{eq:graded-counterexample}.

Let $0\ne h\in\Ann_A(a)$ and set $q=\ord_{\mathfrak m_A}(h)$.
This order is finite because $\mathfrak m_A^{11}=0$.
Multiplication in the associated graded algebra sends the
classes of $h$ and $a$ to the class of $ha$ in
$\mathfrak m_A^{q+5}/\mathfrak m_A^{q+6}$.
Since $ha=0$, it follows that
\[
0=h^*a^*=-\frac14 h^*\ell^5.
\]
The initial form $h^*$ is nonzero, so injectivity in
degrees $0,1,2$ forces $q\geq3$.
This proves $\Ann_A(a)\subseteq\mathfrak m_A^3$.
The argument only uses the implication
$ha=0\Rightarrow h^*a^*=0$; it does not require equality
between the initial ideal of $\Ann_A(a)$ and
$\Ann_C(a^*)$.

Finally, the inverse image of $\Ann_A(a)$ under $R\to A$
is $J:f$, while the inverse image of $\mathfrak m_A^3$
is $\mathfrak m^3+J=\mathfrak m^3$.
Consequently
\[
J:f\subseteq\mathfrak m^3=\overline J,
\]
as claimed. \qedhere
\end{greenrevision}
\end{proof}

\begin{remark}
The example does not contradict Corollary~\ref{cor:hns}. It shows that
the right-hand side cannot be enlarged uniformly to
\(\overline{J_f}\), or to an arbitrary ideal containing
\(\mathfrak m^3\). The local strengthening uses the controlled term
\(\mathfrak m(J_f:f)\).
\end{remark}

\begin{revision}

\end{revision}
\end{document}